\documentclass{elsarticle}
\usepackage{amsmath}
\usepackage{amssymb}
\usepackage{amsfonts}
\usepackage{rotating}
\usepackage{graphicx}
\usepackage{floatflt,epsfig}
\usepackage{lineno,hyperref}
\usepackage{enumerate}
\usepackage{colortbl}
\usepackage{array,tabularx,tabulary,booktabs}
\usepackage{longtable}
\usepackage{multirow}
\usepackage{wrapfig}
\usepackage{subcaption}
\newcolumntype{^}{>{\currentrowstyle}}

\journal{arXiv}
\newcommand{\Z} {{\mathbb Z}}

\newtheorem{definition}{Definition}

\newtheorem{remark}{Remark}
\newtheorem{corollary}{Corollary}
\newtheorem{theorem}{Theorem}
\newtheorem{proposition}{Proposition}
\newcommand{\proof}{\medskip\noindent{\bf Proof.~}}
\begin{document}
\renewcommand{\abstractname}{Abstract}
\renewcommand{\refname}{References}
\renewcommand{\tablename}{Table.}
\renewcommand{\arraystretch}{0.9}
\thispagestyle{empty}
\sloppy

\begin{frontmatter}
\title{Spectra of strongly Deza graphs}

\author[06]{Saieed Akbari}
\ead{s\_akbari@sharif.edu}

\author[08]{Willem~H.~Haemers}
\ead{Haemers@tilburguniversity.edu}

\author[07]{Mohammad Ali Hosseinzadeh}
\ead{hosseinzadeh@ausmt.ac.ir}

\author[03]{Vladislav~V.~Kabanov}
\ead{vvk@imm.uran.ru}

\author[01,04]{Elena~V.~Konstantinova\corref{cor1}}
\cortext[cor1]{Corresponding author}
\ead{e\_konsta@math.nsc.ru}

\author[03,05]{Leonid~Shalaginov}
\ead{44sh@mail.ru}

\address[06]{Department of Mathematical Sciences, Sharif University of Technology, Azadi Street, P.O. Box 11155-9415, Tehran, Iran}
\address[08]{Department of Econometrics and Operations Research, Tilburg University, Tilburg, The Netherlands}
\address[07]{Faculty of Engineering Modern Technologies, Amol University of Special Modern Technologies, Amol 4616849767, Iran}
\address[03]{Krasovskii Institute of Mathematics and Mechanics, S. Kovalevskaja st. 16, Yekaterinburg, 620990, Russia}
\address[01]{Sobolev Institute of Mathematics, Ak. Koptyug av. 4, Novosibirsk, 630090, Russia}
\address[04]{Novosibisk State University, Pirogova str. 2, Novosibirsk, 630090, Russia}
\address[05]{Chelyabinsk State University, Brat'ev Kashirinyh st. 129, Chelyabinsk,  454021, Russia}

\begin{abstract}
A Deza graph $G$ with parameters $(n,k,b,a)$ is a $k$-regular graph with $n$ vertices such that any two distinct vertices have $b$ or $a$ common neighbours. The children $G_A$ and $G_B$ of a Deza graph $G$ are defined on the vertex set of $G$ such that every two distinct vertices are adjacent in $G_A$ or $G_B$ if and only if they have $a$ or $b$ common neighbours, respectively. A strongly Deza graph is a Deza graph with strongly regular children. In this paper we give a spectral characterisation of strongly Deza graphs, show relationships between eigenvalues, and study strongly Deza graphs which are distance-regular.
\end{abstract}

\begin{keyword} Deza graph, eigenvalues, strongly regular graph, divisible design graph, distance-regular graph, cospectral graphs

\vspace{\baselineskip}
\MSC[2010] 05C50\sep 05E10\sep 15A18
\end{keyword}
\end{frontmatter}

\section{Introduction}\label{sec0}
{\em A Deza graph} $G$ with parameters $(n,k,b,a)$ is a $k$-regular graph of order $n$ for which the number of common neighbours of two vertices takes just two values, $b$ or $a$, where $b\geqslant a$. Moreover, $G$ is not the complete graph or the edgeless graph. Deza  graphs were introduced in~\cite{D94}, and the name was given in~\cite{EFHHH}, where the basics of Deza graph theory were founded and different constructions of Deza graphs were presented.

Deza graphs can be considered in terms of matrices. Suppose $G$ is a graph with $n$ vertices, and $M$ is its adjacency matrix. Then $G$ is a Deza graph with parameters $(n,k,b,a)$ if and only if $$M^2=aA+bB+kI$$ for some symmetric $(0,1)$-matrices $A$ and $B$ such that $A+B+I=J$, where $J$ is the all ones matrix and $I$ is the identity matrix. 

Let $G$ be a Deza graph with $M$, $A$ and $B$ as above. If $b=a$ we put $A=J-I$ and $B=O$. Then $A$ and $B$ are adjacency matrices of graphs, and the corresponding graphs $G_A$ and $G_B$ are called  the {\it children} of $G$. A Deza graph of diameter two which is not a strongly regular graph is called a {\em strictly Deza graph}. Deza graphs whose children are complete multipartite graphs and their complement are known as {\em divisible design graphs}. Divisible design graphs were studied in~\cite{CH, HKM}. 

\begin{definition} A Deza graph is called a strongly Deza graph if its children are strongly regular graphs.
\end{definition}

Obviously, divisible design graphs are strongly Deza graphs.

In this paper spectral properties of strongly Deza graphs are studied. The paper is organized as follows. First, spectral relationships between Deza graphs and their children are given. Second, a spectral characterization of strongly Deza graphs is presented. Then, the  relationships between eigenvalues are shown for strongly Deza graphs with four and five distinct eigenvalues. We pay special attention to singular strongly Deza graphs, where a graph is said to be {\it singular} if and only if zero is its eigenvalue.
Finally, we consider distance-regular strongly Deza graphs.

\section{General results}

In this section we present general results on spectral properties of Deza graphs and strongly regular graphs. Spectral relationships between Deza graphs and their children are given by the following result.

\begin{theorem}{\rm\cite[Theorem~3.2]{AGHKKS}} \label{children_spec}
Let $G$ be a Deza graph with parameters $(n,k,b,a)$, $b>a$. Let $M,A,B$ be the adjacency matrices of $G$ and its children, respectively. If $\theta_1=k,\theta_2,\ldots,\theta_n$ are the eigenvalues of $M$, then

{\rm (i)} The eigenvalues of $A$ are 
$$\alpha = \cfrac{b(n-1)-k(k-1)}{b-a},\ \alpha_2 =\cfrac{k-b-\theta_2^2}{b-a},\
  \ldots,\ \alpha_n =\cfrac{k-b-\theta_n^2}{b-a}.$$

{\rm (ii)} The eigenvalues of $B$ are
  $$\beta =  \cfrac{a(n-1)-k(k-1)}{a-b},\ \beta_2 = \cfrac{k-a-\theta_2^2}{a-b},\
  \ldots,\ \beta_n = \cfrac{k-a-\theta_n^2}{a-b}.$$
\end{theorem}

If $b=a$ then $A=J-I$, $B=0$, $G$ is a strongly regular graph and $\theta_i^2=k-b$ for $i=2,\ldots,n$. If $b\neq a$ then children $G_A$ and $G_B$ are regular graphs with degrees $\alpha$ and $\beta$, respectively.

\begin{remark}\label{Remark-Th1} By the proof of~{\rm \cite[Theorem~3.2]{AGHKKS}} it follows that the multiplicities of $\alpha_i$ and $\beta_i$, $2\leqslant i\leqslant n$, are obtained as a summation of the multiplicities of $\pm \theta_i$ except the case  $\theta_n=-k$ with the multiplicity $1$ when a graph is bipartite~{\rm (see \cite[Proposition 3.4.1]{BH})}.
\end{remark}

It is known~\cite{EFHHH} that Deza graphs are a generalisation of strongly regular graphs in such a way that the number of common neighbours of any pair of distinct vertices in a Deza graph does not depend on the adjacency. Moreover, any Deza graph with parameters $(n,k,b,a)$ is a strongly regular graph with parameters $(n,k,\lambda, \mu)$ if and only if $M=A$, $M=B$ or $b=a$. Then we have  $\{\lambda, \mu\}=\{a,b\}$ and 
\begin{equation}\label{M^2}
M^2 = kI+\lambda M + \mu (J-I-M).
\end{equation}

Note that if $b=a$, the children are not strongly regular graphs, and therefore the graph is not a strongly Deza graph. But it is strongly regular with $\lambda=\mu$.

An eigenvalue of the adjacency matrix of a graph is said to be {\em principal} if the all-ones vector is not orthogonal to the associated eigenspace, and {\em restricted} otherwise. The next theorem is well-known (see for example~\cite[Theorem~1.3.1]{BCN89}) and contains a complete information about spectra of strongly regular graphs by their parameters.

\begin{theorem}\label{srg} Let $G$  be a strongly regular graph with parameters $(n,k,\lambda,\mu)$ and the eigenvalues $k$, $r$, and $s$.  Then the following statements hold.

{\rm{(i)}} The principal eigenvalue $k$ has the multiplicity 1.

{\rm{(ii)}} The restricted integer eigenvalues $$r,s=\cfrac{(\lambda-\mu)\pm \sqrt{(\lambda-\mu)^2+4(k-\mu)}}{2}$$ have the multiplicities $$f,g=\cfrac{1}{2} \left(n-1 \mp \cfrac{(r+s)(n-1)+2k}{r-s}\right).$$

{\rm{(iii)}} If $r$ and $s$ are not integers, then
$$r,s=\cfrac{-1\pm\sqrt{n}}{2}$$
with the same multiplicities. 
\end{theorem}

Actually, this theorem can be obtained as a consequence of Theorem~\ref{children_spec} since if $G$ is a strongly regular graph, then $G$ itself and its complement are children of $G$, except the case when $\lambda=\mu$ and the complete graph and its complement are the children. It is known~\cite[Theorem 6.7]{CDS} that the smallest eigenvalue of a graph $G$ is at least $-1$ if and only if $G$ is a vertex disjoint union of complete graphs, and the second largest eigenvalue of a strongly regular graph $G$ is non-positive if and only if $G$ is a complete multipartite graph.  

\section{Spectral characterization}

By Theorems~\ref{children_spec} and~\ref{srg}, a strongly Deza graph has at most three distinct absolute values of its eigenvalues. But we can be more precise.

\begin{proposition}
Suppose $G$ is a strongly Deza graph with parameters $(n,k,b,a)$. \\
{\rm (i)} $G$ has at most five distinct eigenvalues.\\
{\rm (ii)} If $G$ has two distinct eigenvalues, then $a=0$, $b=k-1 \geqslant 1$, and $G$ is a disjoint union of cliques of order $k+1$.\\
{\rm (iii)} If $G$ has three distinct eigenvalues, then $G$ is a strongly regular graph with parameters $(n,k,\lambda ,\mu)$, where $\{\lambda,\mu\}=\{a,b\}$, or $G$ is disconnected and each component is a strongly regular graph with parameters $(v,k,b,b)$, or each component is a complete bipartite graph $K_{k,k}$ with $k\geqslant 2$.
\end{proposition}

\proof {\rm (i)} The eigenvalue of $G$ takes at most three distinct absolute values, one of which equals $k$. If $-k$ is not an eigenvalue, $G$ has at most five distinct eigenvalues. If $-k$ is an eigenvalue, then $G$ is bipartite. Hence, two vertices from different parts of the bipartition always have $a=0$ common neighbours, which implies that $G_B$ is disconnected. Therefore, $G$ is a divisible design graph, which has at most five distinct eigenvalues (see~\cite{HKM}).

{\rm (ii)} Any $k$-regular graph with just two eigenvalues is a disjoint union of cliques of order $k+1$, and if $k\geqslant 2$ it is a strongly Deza graph. 

{\rm (iii)} It is well-known that a connected regular graph with three distinct eigenvalues is a strongly regular graph (see~\cite[Proposition 2.1]{Doob}). If $G$ is disconnected, then $G$ is a disconnected divisible design graph. It follows (see~\cite[proposition 4.3]{HKM}) that each component is a strongly regular graph with parameters $(v,k,b,b)$, or it is a bipartite incidence graph of a symmetric block design with parameters $(v,k,b)$. If $v>k$, the latter option has four distinct eigenvalues, so does not occur. If $v=k$, then $G$ is the disjoint union of complete bipartite graphs, which is a strongly Deza graph if $k \geqslant 2$.
\hfill $\square$\\ 

If $G$ is a bipartite graph, then the~\emph{halved} graphs of $G$ are two connected components of the graph on the same vertex set, where two vertices are adjacent whenever they are at distance two in $G$.

The next theorem gives a spectral characterization of strongly Deza graphs.

\begin{theorem}\label{spec1}
Let $G$ be a connected Deza graph with parameters $(n,k,b,a)$, $b>a$, and it has at most three distinct absolute values of its eigenvalues. 

${\rm (i)}$ If $G$ is a non-bipartite graph, then $G$ is a strongly Deza graph.

${\rm (ii)}$ If $G$ is a bipartite graph, then either $G$ is a strongly Deza graph or its halved graphs are strongly Deza graphs. 
\end{theorem}

\proof ${\rm (i)}$ Since $G$ is a connected non-bipartite Deza graph, $-k$ is not an eigenvalue of $G$. By Theorem~\ref{children_spec}, each of its children $G_A$ and $G_B$ is a regular graph with at most three distinct eigenvalues. Since $G_A$ and $G_B$ are complement of each other at least one of them is a connected graph. Moreover, since any connected regular graph with at most three distinct eigenvalues
is a strongly regular graph or the complete graph, then $G_A$ and $G_B$ are strongly regular graphs. Thus, $G$ is a strongly Deza graph.

${\rm (ii)}$ Since $G$ is a connected bipartite Deza graph, we have $a=0$, and the value $-k$ is an eigenvalue of $G$. If we put $\theta_2=-k$ then by Theorem~\ref{children_spec} we have: $$\beta_2=\cfrac{k-a-\theta_2^2}{a-b}=\cfrac{k-(-k)^2}{-b}=\beta,$$ and $G_B$ is a regular graph with at most three distinct eigenvalues.

Since the multiplicity of $\beta$ is at least $2$, the graph $G_B$ is disconnected. Thus $G_B$ is a union of either strongly regular graphs with the same parameters or complete graphs with the same order. In the latter case, $G_A$ is a regular complete multipartite. Hence, $G$ is a strongly Deza graph. 

Let $G_B$ be a union of $t$ strongly regular graphs with the same parameters. Since $G_A$ is the complement of $G_B$, then $G_A$ is a connected regular graph with four distinct eigenvalues whose spectrum  contains the eigenvalue $-\beta-1$ with the multiplicity $t-1$ (see Section~$4.2$~in~\cite{EvD}). By Theorem~\ref{children_spec} we have:
$$-\beta-1=-\left(\cfrac{a(n-1)-k(k-1)}{a-b}\right)-1=\cfrac{k-b-(-k)^2}{b},$$ 
where the last formula corresponds to the eigenvalue $\alpha_n$ of $G_A$ related to the eigenvalue $-k$ of $G$. By Remark~\ref{Remark-Th1}, the multiplicity of $-k$ equals $1$, hence $t=2$. Since $G$ is a connected bipartite graph each of two halved graphs of $G$ are strongly Deza graphs. \hfill  $\square$\\ 

Spectral properties of regular bipartite graphs with three distinct non-negative eigenvalues were studied by T.~Koledin and Z.~Stani\'c in~\cite{KS}. In particular, the relations between these graphs and two-class partially balanced incomplete block designs were obtained. The interesting observation is that such the design is presented by a matrix equation in which the adjacency matrix of a strongly regular graph is involved by a similar way as it is included in the matrix equation of a strongly Deza graph. Bipartite Deza graphs with six distinct eigenvalues belong to the class studied in~\cite{KS}. Specific bipartite, regular graphs with five eigenvalues were considered by D.~Stevanovi\'c in~\cite{S11}. \\

As immediate consequence of Theorem~\ref{children_spec}, we have the following statement. 

\begin{remark}\label{rem2}
Let $M$ be the adjacency matrix of a strongly Deza graph $G$ with the spectrum $\{k^1,\theta_2^{m_2},\theta_3^{m_3},\theta_4^{m_4},\theta_5^{m_5}\}$, where the exponents denote multiplicities. Then $\theta_2 = - \theta_5, \ \theta_3 = - \theta_4$ and the following equation holds:\begin{equation} \label{TrM}tr(M)=k+(m_2-m_5)\theta_2+(m_3-m_4)\theta_3=0.\end{equation} Also, only one multiplicity can meet zero value, and if so then the corresponding opposite eigenvalue is integer. 
\end{remark}
The trace of the adjacency matrix $M$ of $G$ is presented as follows:  
$$tr(M)=k+m_2 \theta_2 + m_3 \theta_3+m_4 \theta_4 + m_5 \theta_5=0.$$ 
Since  $\theta_2 = - \theta_5, \ \theta_3 = - \theta_4$ then we have equation~(\ref{TrM}). The last statement immediately follows from Theorem~2.6 obtained by E.~R. van Dam in~\cite{EvD}.

The next theorem gives some conditions on an integral strongly Deza graph with respect to eigenvalues of its children.

\smallskip

\begin{theorem}\label{square} Let $G$ be a strongly Deza graph with parameters $(n,k,b,a)$. Let its child $G_A$ be a strongly regular graph with parameters $(n,\alpha,\lambda,\mu)$ and eigenvalues $\alpha,r,s$ with multiplicities $1,f,g$. If $M$ is the adjacency matrix of $G$ with spectrum $\{k^1,\theta_2^{m_2},\theta_3^{m_3},\theta_4^{m_4},\theta_5^{m_5}\}$, then one of the following statements hold:
\smallskip

{\rm{(i)}} $\theta_2^2=k-b-s(b-a)$ and $\theta_3^2=k-b-r(b-a)$ are squares; in this case $G$ is an integral graph.
\smallskip

{\rm{(ii)}} $\theta_2^2=k-b-s(b-a)$ is not a square; then $\theta_3^2=k-b-r(b-a)$ is a nonzero square and $m_2=m_5=f/2$. 
\smallskip

{\rm{(iii)}} $\theta_3^2=k-b-r(b-a)$ is not a square; then $\theta_2^2=k-b-s(b-a)$ is a nonzero square and $m_3=m_4=g/2$.
\end{theorem}

\proof Clearly, $G$ has at most three distinct absolute values of its five eigenvalues. Let $\theta_2 =-\theta_5$, $\theta_3 =-\theta_4$. Since $tr(M)=k+(m_2-m_5)\theta_2+(m_3-m_4)\theta_3 =0$ we have either $m_2 \neq m_5$ or $m_3 \neq m_4$. Without loss of generality we put $$\theta_{2,5}=\pm\sqrt{k-b-s(b-a)}$$ and $$\theta_{3,4}=\pm\sqrt{k-b-r(b-a)}.$$ 

If $G_A$ has no integer eigenvalue except $\alpha$, then by Theorem~\ref{srg} we have: $$r,s=\frac{-1\pm \sqrt{n}}{2}.$$ Let us consider the polynomial $p(x)=\cfrac{m(x)}{x-k}$, where $m(x)$ is the minimal polynomial of $M$.  Thus, using the above formulas we have: $$p(x)=(x-\theta_2)(x+\theta_2)(x-\theta_3)(x+\theta_3)=$$ 
$$=(x^2-k+b+r(b-a))(x^2-k+b+s(b-a))=$$
$$=\left(x^2-k+b+\frac{-1 +\sqrt{n}}{2}(b-a)\right)\left(x^2-k+b+\frac{-1 -\sqrt{n}}{2}(b-a)\right)=$$ 
$$=\left(x^2-k+\frac{b+a}{2}+\frac{\sqrt{n}}{2}(b-a)\right)\left(x^2-k+\frac{b+a}{2}-\frac{\sqrt{n}}{2}(b-a)\right)$$ 
and finally we have:
$$p(x)=\left(x^2-k + \frac{b+a}{2}\right)^2 - \frac{n(b-a)^2}{4}.$$ 
Since $p(x) \in \Z[x]$, if $n$ is not a square, then $p(x)$ is irreducible in $\Z[x]$. Therefore, all multiplicities of the restricted eigenvalues of $G$ are the same. This is a contradiction. 

Assume that the graph $G_A$ is integral. By Theorem~\ref{srg}, the values $\theta_{2,5}$ and $\theta_{3,4}$ are quadratic irrationals. 

Let $\theta_2^2=\tau_2^2\sigma_2$ and $\theta_3^2=\tau_3^2\sigma_3$, where $\sigma_2$ and $\sigma_3$ are square free factors of integers $\theta_2^2$ and $\theta_3^2$, respectively, and the following equation holds: 

\begin{equation} \label{k}
k=(m_5-m_2)\tau_2\sqrt{\sigma_2}+(m_4-m_3)\tau_3\sqrt{\sigma_3}.
\end{equation} 

If $\sigma_2=\sigma_3= 1$ then the statement ${\rm{(i)}}$ holds.

Let $\sigma_2\neq 1$ and $\sigma_3\neq 1$. Since $\sigma_2$ and $\sigma_3$ are square free, then the square of the right part in~(\ref{k}) is an integer, and so $\sqrt{\sigma_2\sigma_3}$ is an integer. Hence, $\sigma_2=\sigma_3$ and $$k=((m_5-m_2)\tau_2+(m_4-m_3)\tau_3)\sqrt{\sigma_2}.$$ 

Since $\sigma_2$ is square free then $\sqrt{\sigma_2}$ is not an integer. Hence, the last equation does not hold, a contradiction.

By~(\ref{k}), if $\sigma_2\neq 1$ and $\sigma_3=1$ then $m_5-m_2$ is zero and case~${\rm{(ii)}}$ holds, and if $\sigma_3\neq 1$ and $\sigma_2=1$ then $m_4-m_3$ is zero, and case~${\rm{(iii)}}$ holds. \hfill $\square$ \\

Theorem \ref{square} is a generalization of Theorem~2.2 in~\cite{HKM}.

\begin{corollary}\label{cor1}
The children of a strongly Deza graph are integral graphs.
\end{corollary}

\section{Eigenvalue relationships}

In this section we present two theorems on the relationships between eigenvalues of a strongly Deza graph with four or five distinct eigenvalues.  

\begin{theorem}\label{theta_4}
If $G$ is a strongly Deza graph with parameters $(n,k,b,a)$ then $G$ has spectrum $\{k^1,\theta_2^{m_2},\theta_3^{m_3},\theta_4^{m_4}, \theta_5^{m_5}\}$ such that there are two opposite integer eigenvalues $\theta_2 = - \theta_5$ and the other eigenvalues are expressed as follows:
\begin{equation} \label{theta_3,4}
\theta_{3,4}=\pm \sqrt{\cfrac{k(n-k)-(m_2+m_5)\theta_2^2}{n-m_2-m_5-1}}.
\end{equation}
Moreover, the multiplicity of one integer eigenvalue can meet zero value.

\end{theorem}

\proof The graph $G$ has at most three distinct absolute values of its eigenvalues one of which is equal to $k$. If $M$ is the adjacency matrix of $G$, then  we have: 

$$tr(M)=k+(m_2-m_5)\theta_2+(m_3-m_4)\theta_3=0.$$

Let $G_A$ be a child of $G$. By Theorem~\ref{children_spec}, it has three distinct eigenvalues presented as follows:

$$\left(\cfrac{b(n-1)-k(k-1)}{b-a},\ \cfrac{k-b-\theta_3^2}{b-a},\ \cfrac{k-b-\theta_2^2}{b-a}\right),$$ 
and so by the proof of Theorem~\ref{square}, it is an integral strongly regular graph. 

By Remark~\ref{Remark-Th1}, their multiplicities are $(1,f,g)$, where $f=m_2+m_5$, $g=m_3+m_4=n-f-1$. Then the trace of the adjacency matrix $A$ of $G_A$ is presented as follows:  

$$tr(A)=\cfrac{b(n-1)-k(k-1)}{b-a}+f\cfrac{k-b-\theta_2^2}{b-a}+(n-f-1)\cfrac{k-b-\theta_3^2}{b-a}=0.$$ 

By Corollary~\ref{cor1}, the children of $G$ are integral strongly regular graphs such that  either $\theta_2$ and $\theta_3$ are both integers or one of them is a quadratic irrational. 

Let $\theta_2$ be an integer eigenvalue of $G$. Straightforward calculations give the following sequence of equalities:
$$(b(n-1)-k(k-1))+f(k-b-\theta_2^2)+(n-f-1)(k-b-\theta_3^2)=0.$$ 
$$b(n-1)-k(k-1)+(k-b)(n-1)+f(-\theta_2^2)-(n-f-1)\theta_3^2=0.$$ 
$$(n-f-1)\theta_3^2=b(n-1)-k(k-1)+(k-b)(n-1)-f\theta_2^2.$$ 
$$(n-f-1)\theta_3^2=k(n-k)-f \theta_2^2.$$ 
$$\theta_3^2=\cfrac{k(n-k)-f\theta_2^2}{n-f-1}.$$ 

\noindent Since $\theta_4=-\theta_3$, we have: 
$$\theta_{3,4}=\pm \sqrt{\cfrac{k(n-k)-f \theta_2^2}{n-f -1}},$$
 which gives us~(\ref{theta_3,4}). The last statement follows from Remark \ref{rem2}. \hfill $\square$\\ 

More meaningful relationships between the eigenvalues of strongly Deza graphs with four distinct eigenvalues are given by the following result. 

\begin{theorem}\label{4eigenvalues}
Let $G$ be a strongly Deza graph with parameters $(n,k,b,a)$ and spectrum $\{k^1,\theta_2^{m_2},\theta_3^{m_3},\theta_4^{m_4}\}$ and let $\theta_4=-\theta_3$, $m_3=m_4$. Then 
\begin{equation} \label{four_theta_3,4}
\theta_{3,4}=\pm  \sqrt{k\left(1+\cfrac{(\theta_2 +1)(m_2 +1)}{n-m_2 -1}\right)},
\end{equation}
where $\theta_2=-k/m_2$ is an integer and $\theta_2\leqslant -1.$
\end{theorem}

\proof Since $\theta_4=-\theta_3$ and $m_3=m_4$ we have: 
$$tr(M)=k+m_2\theta_2=0.$$ 
Hence, $k=-m_2\theta_2$ with integer $\theta_2$ which is at most $-1$.  

By Theorem~\ref{theta_4}, the other eigenvalues of $G$  are expressed as follows:
$$\theta_{3,4}=\pm \sqrt{\cfrac{k(n-k)-m_2\theta_2^2}{n-m_2-1}}.$$
Since $k=-m_2\theta_2$ then straightforward calculations give the following: 
$$\theta_{3,4}=\pm \sqrt{k \left(1+\cfrac{m_2\theta_2+m_2+1+\theta_2}{n-m_2 -1}\right)}.$$ 
Finally, we have: 
$$\theta_{3,4}=\pm  \sqrt{k\left(1+\cfrac{(\theta_2 +1)(m_2 +1)}{n-m_2 -1}\right)},$$
which completes the proof. \hfill $\square$\\

Theorem~\ref{children_spec} allows to calculate the spectra of the children $G_A$ and  $G_B$ of a strongly Deza graph $G$. Since the children of $G$ are strongly regular graphs, their parameters are calculated from their spectra. Hence, the parameters of $G$ determine the parameters of $G_A$ and $G_B$. However, the parameters of $G_A$ and $G_B$ do not allow to calculate the parameters of $G$.

\section{Four distinct eigenvalues}

Connected regular graphs with four distinct eigenvalues were studied by E.~R.~van Dam in~\cite{EvD}. He has found some properties and feasibility conditions of the eigenvalues of such graphs. In the case of strongly Deza graphs with four distinct eigenvalues we found some new properties.

\begin{theorem}\label{singular}
Any singular strongly Deza graph is an integral graph with four distinct eigenvalues.
\end{theorem}

\proof  Let $M$ be the adjacency matrix of a strongly Deza graph $G$. Thus, $M$  has at most five distinct eigenvalues $\{k^1,\theta_2^{m_2},\theta_3^{m_3},\theta_4^{m_4},\theta_5^{m_5}\}$ with $\theta_2= -\theta_5, \theta_3=-\theta_4$. Suppose $\theta_2=0$. Then we have: $$tr(M)=k+(m_3-m_4)\theta_3=0.$$ By~\cite[Theorem~2.6]{EvD}, if $\theta_3$ is not an integer eigenvalue, then $m_3=m_4$. Hence, $k=0$ and we have a contradiction. \hfill $\square$\\


There are infinitely many singular strongly Deza graphs~\cite[Theorem 2]{KSh} arising from the affine group $\mathbb{A}\mathrm{ff}(1,\mathbb F_{q^t})$, for any prime power $q$ and $t>1$, as divisible design Cayley graphs with parameters $(v,k,\lambda_1, \lambda_2,m,n)$, where
$$v=q^t (q^t-1)/(q-1),\quad k=q^{t-1}(q^t-1),$$
$$\lambda_1=q^{t-1}(q^t-q^{t-1}-1),\quad \lambda_2=q^{t-2}(q-1)(q^t-1),$$ $$m=(q^t-1)/(q-1),\quad n= q^t.$$ 
By~\cite[Theorem 2.2]{HKM}, their eigenvalues are  
$$\theta_{2,5}=\pm\sqrt{k-\lambda_1}=\pm\sqrt{q^{t-1}(q^t-1)-q^{t-1}(q^t-q^{t-1}-1)}=\pm q^{t-1},$$
$$\theta_3=\theta_4=\pm\sqrt{k^2-\lambda_2 v}=0.$$

The smallest example is known as the line graph of the octahedron. It has parameters $(12,6,3,2)$ and spectrum $\{6^1,2^3,0^2,(-2)^6\}$.

\begin{theorem}\label{last} 
Let $G$ be a strongly Deza graph with parameters $(n,k,b,a)$ whose spectrum is $\{k^1,\theta_2^{m_2},\theta_3^{m_3},\theta_4^{m_4}\}$, where $\theta_3 =-\theta_4$. If $m_3=m_4$, then $m_2 \theta_2=-k$ and one of the following statement holds.

${\rm (i)}$ $m_2 =1$, $\theta_2=-k$ and $G$ is a bipartite graph with spectrum
$$\left\{k^1,\sqrt{\cfrac{k(n-2k)}{n-2}}^{\frac{n-2}{2}},-\sqrt{\cfrac{k(n-2k)}{n-2}}^{\frac{n-2}{2}},(-k)^1\right\}.$$ 

${\rm (ii)}$  $m_2 =k$, $\theta_2=-1$ and the spectrum of $G$ is 
$$\left\{k^1,\sqrt{k}^{\frac{n-k-1}{2}},-1^{k},-\sqrt{k}^{\frac{n-k-1}{2}}\right\}.$$ 

${\rm (iii)} $ $m_2 <k$, $\theta_2<-1$.
\end{theorem}

\proof We have $tr(M)=k+m_2\theta_2=0$ and so $m_2 \theta_2=-k$.

First suppose that $m_2=1$ and $\theta_2=-k$. By Theorem~\cite[Theorem 3.4]{CDS}, $G$ is bipartite.  By Theorem~\ref{4eigenvalues}, we have:
\begin{equation} 
\theta_{3,4}=\pm  \sqrt{k\left(1+\cfrac{(\theta_2 +1)(m_2 +1)}{n-m_2 -1}\right)}=\pm  \sqrt{\cfrac{k(n-2k)}{n-2}},
\end{equation}
with the multiplicities $\frac{n-2}{2}$ which gives case~${\rm (i)}$. 

Next assume that $m_2=k$ and $\theta_2= -1$, then we have the eigenvalues
\begin{equation} 
\theta_{3,4}=\pm  \sqrt{k\left(1+\cfrac{(\theta_2 +1)(m_2 +1)}{n-m_2 -1}\right)}=\pm\sqrt{k}
\end{equation}
with the multiplicities $\frac{n-k-1}{2}$ which gives case~${\rm (ii)}$. The proof of case~${\rm (iii)}$ is now clear.
\hfill $\square$\\

\begin{remark} If $G$ is a non-integral graph, then by Theorem~\ref{square} the equality of multiplicities holds for some opposite eigenvalues.
\end{remark}

\begin{remark} If $G$ has a prime degree in Theorem~\ref{last}, then only cases~${\rm (i)}$ and ${\rm (ii)}$ are valid.
\end{remark}

\section{Distance-regular Deza graphs}

Now we investigate when a distance-regular graph can be a strongly Deza graph. For background on distance-regular graphs, we refer to the book~\cite{BCN89} and the survey~\cite{DKT}.

Throughout this section $G=(V,E)$ is a distance-regular graph of order $n$ and diameter $d$. This means that $G$ is a connected graph and there exists numbers $\{a_i,b_i,c_i\}$ such that for every $i\in\{0,\ldots,d\}$ and for every pair of vertices $x\in V$ and $y\in V$ at mutual distance $i$ the following holds: \\
- the number of vertices adjacent to $y$ at distance $i+1$ from $x$ equals $b_i$;\\
- the number of vertices adjacent to $y$ at distance $i$ from $x$ equals $a_i$;\\
- the number of vertices adjacent to $y$ at distance $i-1$ from $x$ equals $c_i$.

The numbers $a_i,b_i,c_i$ are called the {\em intersection numbers} of $G$. It follows that $G$ is a regular graph of degree $k=b_0=a_i+b_i+c_i$, $i=1,\ldots,d$. Also, for $i=0,\ldots,d$ the number $k_i$ of vertices at distance $i$ from a given vertex only depends on $i$. For $i=1,\ldots d$, we define $E_i$ to be the set of vertex pairs from $G$ at mutual distance $i$. Then $E_1=E$ and $(V,E_i)$ is a regular graph of degree $k_i$, which is called {\it the distance-$\rm i$} graph. We say that $G$ is {\em antipodal} if the distance-$\rm d$ graph is a disjoint union of cliques.

If $d=2$, then $G$ is a strongly regular graph with parameters $(n,k,a_1,c_2)$. If $d\geqslant 3$, then two distinct vertices of $G$ have $a_1$, $c_2$, or $0$ common neighbours, and $0$ does occur. Clearly $c_2\neq 0$, so $G$ is a Deza graph if $a_1=0$, which means that $G$ is triangle free, or $a_1=c_2$. Thus, we have the following statement.

\begin{proposition}~\label{DRDG}
A distance-regular graph $G$ of diameter $d\geqslant 3$ is a Deza graph if and only if one of the following holds: \\
{\rm (i)} $a_1=0$; then $G$ has parameters $(n,k,c_2,0)$ and children $(V,E_2)$ and its complement. \\
{\rm (ii)} $a_1=c_2$; then $G$ has parameters $(n,k,c_2,0)$ and children $(V,E_1\cup E_2)$ and its complement.
\end{proposition}

There exist many distance-regular graphs, including all bipartite ones, with the above properties, and obviously none of these will be a strictly Deza graph. But some are strongly Deza graphs. From Proposition~\ref{DRDG} it follows that $G$ is a strongly Deza graph whenever $a_1=0$ and $(V,E_2)$ is strongly regular, or $a_1=c_2$ and $(V,E_1\cup E_2)$ is strongly regular. It is known that the intersection numbers of $G$ determine the eigenvalues of $(V,E_1)$ and $(V,E_1\cup E_2)$. This implies that the property {\lq}being a distance-regular strongly Deza graph{\rq} is determined by the intersection numbers. If the children are a complete multipartite graph and its complement, then $G$ is a divisible design graph. Distance-regular divisible design graphs have been classified by the following result.

\begin{theorem}~\label{DRDDG} {\rm \cite[Theorem 4.14]{HKM}}
A graph $G$ is a distance-regular divisible design graphs if and only if $G$ is one of the following:\\
{\rm (i)} a complete multipartite graph;\\
{\rm (ii)} the incidence graph of a symmetric $2$-design;\\
{\rm (iii)} an antipodal distance-regular graph of diameter $3$ with $a_1 = c_2$.
\end{theorem}

A complete multipartite graph has diameter $2$. The incidence graph of a symmetric $2$-design is the same as a bipartite distance-regular graph of diameter $3$. It belongs to case~(i) of Theorem~\ref{last}. For more about case~(iii) of Theorem~\ref{DRDDG}, we refer to~\cite[p.~431]{BCN89}; see also Corollary~\ref{DDG_DRG} below. By Theorem~\ref{DRDDG}, a distance-regular divisible design graph has diameter at most $3$. If a distance-regular strongly Deza graph $G$ is not a divisible design graph, then $d\leqslant 4$. Indeed, if $d\geqslant 5$ then one of the children of $G$ has diameter at least 3, hence it cannot be a strongly regular graph. It also follows because $G$ has at most five distinct eigenvalues. We do not know if there exist examples with $d=4$, but for $d=3$ there is an infinite family, known as the unitary nonisotropics graph (see~\cite[Theorem~12.4.1]{BCN89}), which has order $q^2(q^2-q+1)$ and degree $q(q-1)$ for prime power $q>2$. It satisfies $c_2=a_1=1$, and belongs to case~(ii) of Proposition~\ref{DRDG}. The child $G_A=(V,E_3)$ is the block graph of a Steiner $2$-$(q^3+1,q+1,1)$-design which is the classical unital design. Thus, one can conclude the following.

\begin{theorem}~\label{UNG}
For every prime power $q>2$, there exists a strongly Deza graph with parameters $(q^2(q^2-q+1),q(q-1),1,0)$ and spectrum
\[
\{q(q-1),\ (q)^{(q^4-q)/2-q^3+q^2},\ (-1)^{q^3},\ (-q)^{(q^4-q)/2-q^3+q-1}\}.
\]
The distance-$\rm 3$ child $G_A$ is a strongly regular graph with parameters \[(q^2(q^2-q+1),(q-1)(q+1)^2,2q^2-2,(q+1)^2).\]
\end{theorem}

\begin{remark}
M.~A.~Fiol~{\rm~\cite{F}} uses the name {\lq}strongly distance-regular{\rq} if $(V,E_d)$ is strongly regular. So, if $d=3$ and $a_1=c_2$ then a distance-regular strongly Deza graph is a strongly distance-regular graph and vice versa. In particular, the graphs of Theorem~\ref{UNG} are strongly distance-regular. Worth mentioning is that a distance regular graph of diameter $3$ is strongly distance-regular if and only if $-1$ is an eigenvalue {\rm (}see also~{\rm \cite[Proposition~4.2.17(ii)]{BCN89}}{\rm)}.
\end{remark}

Note that case~(ii) of Theorem~\ref{DRDDG} satisfies $a_1=0$, so the graphs belong to case~(i) of Proposition~\ref{DRDG}. However, we don't know an example of case~(i) of Proposition~\ref{DRDG} which is not a divisible design graph. But for $d=3$ there do exist feasible intersection numbers for the required distance-regular graphs. For example, $(n,k,k_2,c_2)=(210,11,110,1)$ and $(320,22,231,2)$ are both feasible~(see~\cite[Chapter~14]{BCN89}).

\subsection{Cospectral graphs}

It is well-known that a distance-regular graph with diameter $d$ has $d+1$ distinct eigenvalues, and that the intersection numbers determine the spectrum. Also the converse is true: the spectrum of a distance-regular graph determines the intersection numbers. This, however, does not mean that a graph with the same spectrum as a distance-regular graph has to be a distance-regular graph. Indeed, there exist many graphs cospectral with (i.~e. with the same spectrum as) a distance-regular graph which are not distance-regular. However, as it is given by the result below, there are special situations for which it is true. 

\begin{theorem}{\rm \cite{H}}~\label{SCDR}
Suppose $G'$ is a graph cospectral with a distance-regular graph $G$. In the following cases $G'$ is also distance-regular with the same intersection numbers as $G$: \\
${\rm (i)}$ $G$ has diameter $2$, i.~e. $G$ is a strongly regular graph; \\
${\rm (ii)}$ $G$ is bipartite and $d=3$, i.~e. $G$ is the incidence graph of a symmetric $2$-design; \\
${\rm (iii)}$ $c_2=1$.
\end{theorem}

See~{\rm \cite[Section 10.1]{DKT}} for a longer list. Theorem~\ref{SCDR} applies to the graphs from Theorem~\ref{DRDDG} (i) and (ii) and Theorem~\ref{UNG}. Thus, for these distance-regular graphs cospectral graphs are distance-regular with the same intersection numbers.
But for case~(iii) of Theorem~\ref{DRDDG} there do exist cospectral graphs which are not distance-regular (see below). The question is whether such a graph can still be a Deza graph. Then Theorem~\ref{spec1} implies that it is in fact a strongly Deza graph. This, however, is not likely to happen because of the following results. 

\begin{theorem}\label{CosDR} {\rm \cite{H}}
Suppose $G'$ is a graph cospectral with is a distance-regular graph $G$ of diameter $3$. If each vertex in $G'$ has the same number of vertices at distance $3$ as a vertex in $G$, then $G'$ is also distance-regular with the same intersection numbers as $G$.
\end{theorem}

\begin{remark}\label{spectral_excess}
It is not necessary that the distance-regular graph $G$ really exists, it suffices that the spectrum corresponds to a feasible intersection array. In fact, the above theorem has been improved further with weaker spectral conditions and generalized to arbitrary diameter. The general version is now known as {\lq}spectral excess theorem{\rq}; see~{\rm \cite[Section 10.3]{DKT}}.
\end{remark}

\begin{proposition}\label{CosDeza}
Suppose $G'$ is a graph cospectral with a distance-regular strongly Deza graph $G$ with diameter $3$ and $a_1=c_2$. If $G'$ is a Deza graph then $G'$ is a strongly Deza graph and one of the following holds: \\
{\rm (i)} $G'$ is distance-regular with the same intersection numbers as $G$, \\
{\rm (ii)} $G'$ and $G$ have different parameters $a$ and $b$ as Deza graphs.
\end{proposition}

\proof Assume $G'$ is a Deza graph with parameters $(n,k,b,a)$. Since $G$ and $G'$ are cospectral, both graphs have order $n$ and degree $k$. Moreover, $G'$ satisfies the conditions of Theorem~\ref{spec1}~{\rm (i)}, therefore $G'$ is a strongly Deza graph.

If case~{\rm (ii)} does not hold, then $G$ and $G'$ have the same parameters, so $a=0$ and $b=a_1$. Clearly $G$ has $nka_1/6$ triangles, and since cospectral graphs have the same number of triangles, $G'$ also has $nka_1/6 = nkb/6$  triangles. This implies that every edge of $G'$ belongs to $b$ triangles. Hence, any two vertices in $G'$ with $a=0$ common neighbours are at distance $3$. This means that for both graphs $G$ and $G'$ the children $G_A$ and $G'_A$ are the distance-$3$ graphs. Since $G$ and $G'$ have the same parameters, $G_A$ and $G'_A$ have the same degree and hence $G$ and $G'$ have the same number of vertices at distance $3$ from a given vertex. So, by Theorem~\ref{CosDR}, $G'$ is distance-regular with the same intersection numbers as $G$.
\hfill $\square$ \\

If case~(ii) of Proposition~\ref{CosDeza} occurs, then the children of $G$ have different parameters and spectrum from those of $G'$,
but Theorem~\ref{children_spec} implies that the eigenvalue multiplicities are the same. Pairs of strongly regular graphs with different spectrum but the same multiplicities (other than the spectrum of the complement) exist (for example the Petersen graph and the complete multipartite graph $K_{2,2,2,2,2}$), but we don't know a pair of cospectral strongly Deza graphs for which the children have this property. Fortunately, often case~{\rm (ii)} can be easily excluded in particular cases. For example if $G$ and $G'$ are both divisible design graphs.

\begin{corollary}\label{DDG_DRG}
Let $G'$ be a divisible design graph with spectrum
\[
\{k^1,\sqrt{k}^m,(-1)^k,(-\sqrt{k})^m\}.
\]
If $k^2-1$ is divisible by $n=2m+k+1$, then $G'$ is an antipodal distance-regular graph with $d=3$ and $a_1=c_2=(k^2-1)/n$.
\end{corollary}
\begin{proof}
According to~\cite[p.431]{BCN89} there exist feasible intersection numbers satisfying $a_1=c_2=(k^2-1)/n$ for a putative distance-regular graph $G$ cospectral with $G'$. As mentioned in Remark~\ref{spectral_excess}, Theorem~\ref{CosDR} and Proposition~\ref{CosDeza} apply even if existence of $G$ is not established.

The children of a divisible design graph are the complete multipartite graph and its complement. The spectrum of the complete multipartite graph $K_{m,\ldots,m}$ of order $n$ equals $\{(n-m)^1,\ 0^{n-n/m},\ (-m)^{n/m-1}\}$. So, clearly the eigenvalue multiplicities determine $m$ and $n$ and therefore case~(ii) of Proposition~\ref{CosDeza} does not occur and the result follows.
\end{proof} \hfill $\square$ \\

\begin{table}
\centering
\begin{tabular}{|l|c|l|r|}
\hline
Name & n & Spectrum & Cosp~~~\\[2pt]
\hline
Icosahedron & 12 & $\{5^1,(\sqrt{5})^3,(-1)^5,(-\sqrt{5})^3\}$  & 1 \\[2pt]
Line\ graph\ of\ Petersen\ graph\! & 15 &  $\{4^1,2^5,(-1)^4,(-2)^5\}$  & 1 \\[2pt]
Johnson\ graph\ $J(6,3)$ & 20 & $\{9^1, 3^5,(-1)^9,(-3)^5\}$ & 6 \\[2pt]
Klein\ graph   & 24 & $\{7^1,(\sqrt{7})^8,(-1)^7,(-\sqrt{7})^8\}$ & 10 \\[2pt]
Taylor\ graph\ of\ Payley(13) & 28 & $\{13^1,(\sqrt{13})^7,(-1)^{13},(-\sqrt{13})^7\}$\! & $\geqslant\!{1173}$ \\[2pt]
\hline
\end{tabular}
\caption{Distance-regular strongly Deza graphs with $a_1=c_2$, $d=3$ and $n\leqslant 30$}\label{DRG}
\end{table}

E.~R.~van~Dam, W.~H.~Haemers, J.~H.~Koolen, and E.~Spence~\cite{DHKS} studied graphs cospectral with distance-regular graphs. They checked the non-trivial distance-regular graphs with at most $70$ vertices and diameter $d\geqslant 3$. The trivial cases are the cycles and the incidence graphs of a trivial symmetric $2$-$(k+1,k,k-1)$ design. Let us look at the non-trivial examples on at most 30 vertices.
There are precisely $29$ such distance-regular graphs, and $19$ of them are strongly Deza graphs. Among these strongly Deza graphs there are $14$ incidence graphs of non-trivial symmetric $2$-designs, which are divisible design graphs corresponding to case~(ii) of Theorem~\ref{DRDDG}. We saw that for these graphs all cospectral graphs are also incidence graphs of a symmetric design.
The remaining five are divisible design graphs corresponding to case~(iii) of Theorem~\ref{DRDDG}. These five graphs with their spectra are given in Table~\ref{DRG}. For three of them there exist several cospectral graphs. The last column named {\lq}Cosp{\rq} in the Table gives the number of non-isomorphic graphs with the given spectrum. For each spectrum there is only one distance-regular graph. For the other graphs option~(ii) of Proposition~\ref{CosDeza} does not occur, simply because there exist no strongly regular graphs with different parameters as the children of $G$ but with the same eigenvalue multiplicities. So, one can conclude that each of these graphs is not a Deza graph.

\section*{Acknowledgements}  The research work of Elena V.~Konstantinova is supported by Mathematical Center in Akademgorodok, the agreement with Ministry of Science and High Education of the Russian Federation number 075-15-2019-1613. The research work of M.~A.~Hosseinzadeh has been supported by a research grant from the Amol University of Special Modern Technologies, Amol, Iran.

\end{document}